\documentclass[11pt]{article}
\usepackage{amsmath}
\usepackage{amssymb}
\usepackage{amsthm}
\usepackage{enumerate}
\usepackage{color}
\usepackage{xfrac}
\usepackage{bbm}
\usepackage{hyperref}
\usepackage{mathrsfs}
\usepackage{enumerate}

\newtheorem{thm}{Theorem}[section]

\newtheorem{lemma}[thm]{Lemma}

\newtheorem{rem}[thm]{Remark}
\newtheorem{cor}[thm]{Corollary}

\begin{document}

\title{Inequalities for ratios of Gamma functions}

\author{Henrik L. Pedersen}

\date{\today}
\maketitle

\begin{abstract}
We prove that the function 
$$
\log \frac{\Gamma(x+a+b)\Gamma(x)}{\Gamma(x+a)\Gamma(x+b)}-b\log\left(1+\frac{a}{x}\right)
$$
is a generalized Stieltjes function of order $3$ for $a>0$ and $0<b<1$. A similar result is obtained when $a>0$ and $b>1$.
\end{abstract}


\section{Introduction and results}
The motivation for this paper stems partly from \cite{DSS}, in which the authors study ratios of Beta functions. Numerous estimates and inequalities are obtained, and in particular it is proved that the inequality  (\cite[Theorem 2.2]{DSS}) 
\begin{equation}
\label{eq:paper}
\frac{B(\beta,y)}{B(\alpha,y)}\leq \left(\frac{\alpha}{\beta}\right)^y
\end{equation}
holds for $0<\alpha<\beta$ and $0\leq y\leq 1$.
The purpose of the present paper is to strengthen this inequality by investigating the difference between the logarithm of these two expressions.

We shall relate the expressions in \eqref{eq:paper} to ratios of Gamma functions. 
The Beta function $B$ is defined as 
$$
B(x,y)=\int_0^1t^{x-1}(1-t)^{y-1}\, dt, \ x,y>0,
$$
and the fundamental relation to Euler's Gamma function is
$$
B(x,y)=\frac{\Gamma(x)\Gamma(y)}{\Gamma(x+y)}.
$$ Putting $x=\alpha$, $a=\beta-\alpha$, $b=y$ the inequality \eqref{eq:paper} can be restated as
\begin{equation}
\label{eq:rewrite}
  \frac{\Gamma(x+a+b)\Gamma(x)}{\Gamma(x+a)\Gamma(x+b)}\geq \left(1+\frac{a}{x}\right)^b  
\end{equation}
for all $x>0$, $a>0$ and $b\in [0,1]$. Several authors have investigated this kind of ratios of Gamma functions. In one of the first papers \cite{BI} it was obtained that the left hand side of \eqref{eq:rewrite} is a (logarithmically) completely monotonic function for $a,b>0$. We shall strengthen this result in Corollary \ref{cor:lcm}. The logarithmic complete monotonicity of the left-hand side of \eqref{eq:rewrite} was later strengthened in \cite{BKP}, where it was obtained  that the logarithm of the left-hand side is a so-called generalized Stieltjes function of order $2$. More general ratios have been investigated in \cite{KarpPrilepkina} and also in the framework of entire functions of finite order, see \cite{AP}.

We shall also formulate our results in terms of generalized Stieltjes functions of positive order. For the reader's convenience these and related classes are introduced below.

A completely monotonic function $g$ is an infinitely often differentiable function defined on $(0,\infty)$ such that $(-1)^ng^{(n)}\geq 0$ on $(0,\infty)$ for all $n\geq 0$. Bernstein's theorem states that $g$ is completely monotonic if and only if there exist a positive measure $\mu$ on $[0,\infty)$ and $c\geq 0$ such that, for $x>0$, 
$$
g(x)=c+\int_0^{\infty}e^{-xt}\, d\mu(t)=c+\mathcal L(\mu)(x).
$$
(Here, $\mathcal L$ denotes the Laplace transform.) A function $h:(0,\infty)\to (0,\infty)$ is called logarithmically completely monotonic if $-h'/h$ is completely monotonic. The class of logarithmically completely monotonic functions  was introduced and characterized by Horn, see \cite{H} and it is a subclass of the completely monotonic functions.

A function $f:(0,\infty)\to \mathbb R$ is called a generalized Stieltjes function of order $\lambda>0$ if
$$
f(x)=c+\int_0^{\infty}\frac{d\nu (t)}{(t+x)^{\lambda}}, \quad x>0,
$$
for a real non-negative constant $c$, and a positive measure $\nu$ on $[0,\infty)$ such that the integral converges for all $x>0$. This class is denoted by $\mathcal S_{\lambda}$. It is well-known (and not difficult to prove) that $f\in \mathcal S_{\lambda}$ if and only if
$$
f(x)=c+\mathcal L(t^{\lambda-1}\varphi(t))(x)
$$
for some $c\geq 0$ and a completely monotonic function $\varphi$. From this characterization it follows that $\mathcal S_{\lambda}\subseteq \mathcal S_{\mu}$ for $0<\lambda\leq \mu$. For an introduction to completely monotonic functions and the Laplace transform, see e.g.\ \cite{W} and \cite{S}. More information on generalized Stieltjes functions can be found in e.g.\ \cite{Sokal} and \cite{KP}.

We define 
\begin{equation}
    \label{eq:phi}
\phi_{a,b}(x)=\log \frac{\Gamma(x+a+b)\Gamma(x)}{\Gamma(x+a)\Gamma(x+b)}- b\log \left(1+\frac{a}{x}\right), \quad a,b>0.
\end{equation}
The relation \eqref{eq:paper} can be cast as $\phi_{a,b}(x)\geq 0$ for $x>0$, $a>0$ and $b\in [0,1]$. We shall strengthen this in Theorem \ref{thm:main}. (Notice that $\phi_{a,1}\equiv 0$.)
\begin{thm}
    \label{thm:main} The following assertions hold. 
    \begin{enumerate}[(a)]
        \item The function $\phi_{a,b}$ is a generalized Stieltjes function of order 3 for $a>0$ and $0< b<1$.  
    \item The function $-\phi_{a,b}$ is a generalized Stieltjes function of order 3 for $a>0$ and $b>1$. 
    \end{enumerate}
\end{thm}
Introducing the function 
$$
R_{a,b}(x)=\frac{\Gamma(x+a+b)\Gamma(x)x^b}{\Gamma(x+a)\Gamma(x+b)(x+a)^b},
$$
we notice the following corollary.
\begin{cor}
\label{cor:lcm}
    For $a>0$ and $0<b<1$, $R_{a,b}$ is logarithmically completely monotonic; for $a>0$ and $b>1$, $1/R_{a,b}$ is logarithmically completely monotonic. 
\end{cor}
The proof of this corollary consists of noticing that by Theorem \ref{thm:main}, $\phi_{a,b}$ is completely monotonic for $0<b<1$, and $-\phi_{a,b}$ is completely monotonic for $b>1$.

In \cite[Proposition 4.1]{BKP} is it proved that 
\begin{align*}
\log \frac{\Gamma(x+a+b)\Gamma(x)}{\Gamma(x+a)\Gamma(x+b)}&=\int_0^{\infty}e^{-xt}\varphi_{a,b}(t)\, dt\\
&=\int_0^{\infty}\frac{g_{a,b}(t)}{(t+x)^2}\, dt,
\end{align*}
where 
\begin{align*}
\varphi_{a,b}(t)&=\frac{(1-e^{-at})(1-e^{-bt})}{t(1-e^{-t})}
\end{align*}
and
\begin{align*}
g_{a,b}(t)&
=\sum_{k=0}^{[t]}\left(1_{(0,a)}\ast1_{(0,b)}\right)(t-k)
\end{align*}
for all positive parameters $a$ and $b$. (Here, $1_M$ denotes the indicator function of the set $M$.) 
In particular, 
$$
\log \frac{\Gamma(x+a+b)\Gamma(x)}{\Gamma(x+a)\Gamma(x+b)}\in \mathcal S_2.
$$
Furthermore, 
$$
b\log \left(1+\frac{a}{x}\right)=b\int_0^{\infty}\frac{h_a(t)}{(t+x)^2}\, dt,
$$
with 
$$
h_a(t)=\left\{ 
\begin{array}{ll}
t, &0\leq t \leq a,\\
a,&t\geq a.
\end{array}
\right.
$$
From the representations above we have 
\begin{equation}
\label{eq:phi-int}
\phi_{a,b}(x)=\int_0^{\infty}\frac{g_{a,b}(t)-bh_a(t)}{(t+x)^2}dt.
\end{equation}
Figure \ref{fig:graph3} shows the graph of $g_{a,b}-bh$ for a specific choice of parameters $a$ and $b$.
\begin{figure}
    \centering
    \includegraphics[width=0.5\linewidth]{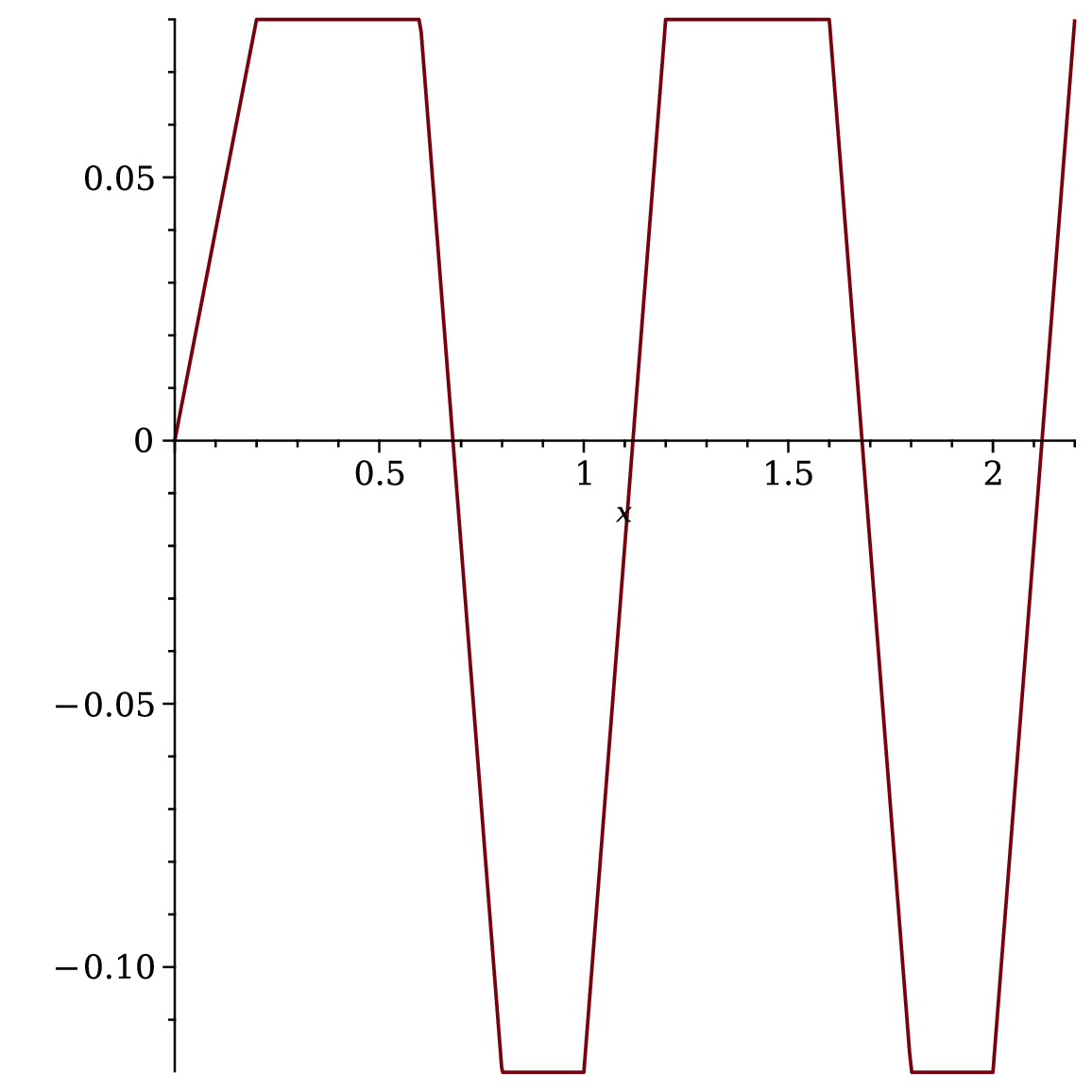}
    \caption{$g_{a,b}-bh$ for $a=0.2$ and $b=0.6$}
    \label{fig:graph3}
\end{figure}
It illustrates that in general the function $g_{a,b}-bh_a$ may change its sign on the positive line so $\phi_{a,b}$ does not belong to $\mathcal S_2$. As a difference of two generalized Stieltjes functions the monotonicity properties of $\phi_{a,b}$ are not directly obtainable, but as Theorem \ref{thm:main} shows it is possible control the cancellation. 

The asymptotic behavior of the Gamma function term is well-known (and can be obtained by partial integration):
$$
\log \frac{\Gamma(x+a+b)\Gamma(x)}{\Gamma(x+a)\Gamma(x+b)}=\int_0^{\infty}e^{-xt}\varphi_{a,b}(t)\, dt=\frac{ab}{x}+\frac{ab(1-a-b)}{2x^2}+\cdots.
$$
We have also 
$$
b\log\left(1+\frac{a}{x}\right)=\frac{ab}{x}-\frac{ba^2}{2x^2}+\cdots,
$$
and this shows that $x^2\phi_{a,b}(x)\to ab(1-b)/2$ as $x\to \infty$. Hence, the elementary function $b\log(1+a/x)$ approximates the first term in \eqref{eq:phi} up to order 2, as $x$ tends to infinity, and their difference is a generalized Stieltjes function of order 3. At the same time the difference has (only) a logarithmic singularity at the origin.

\section{Preliminary results}
In the proof of Theorem \ref{thm:main} some elementary lemmas are needed, and these are given in this section. It is convenient to introduce two auxiliary functions, namely
\begin{align*}
   \xi_{\alpha,\beta}(x)&=\alpha \log(1+\beta/x)-\beta\log(1+\alpha/x)\\
   \eta_{\alpha,\beta}(x)&=\log(1+\alpha/x)-\log(1+\alpha/(x+\beta))
\end{align*}
\begin{lemma}
    \label{lemma:xi}
    The function $\xi_{\alpha,\beta}$ belongs to $\mathcal S_2$ when $0\leq \beta\leq \alpha$. The function $\eta_{\alpha,\beta}$ belongs to $\mathcal S_2$ when $0\leq \alpha, \beta$. 
\end{lemma}
\begin{proof}
Notice that
\begin{align*}
\xi_{1,c}(x)=\log\left(1+\frac{c}{x}\right)-c\log\left(1+\frac{1}{x}\right)=\int_0^{\infty}\frac{h_c(t)-ch_1(t)}{(t+x)^2}\, dt.
\end{align*}
When $0\leq c\leq 1$ we have $h_c(t)-ch_1(t)\geq 0$ for $t\geq 0$ so $\phi_{1,c}\in \mathcal S_2$ for $0\leq c\leq 1$. For the general case where $0\leq \beta\leq \alpha$ we observe that
$\xi_{\alpha,\beta}(x)=\alpha\xi_{1,\beta/\alpha}(x/\alpha)$.

Concerning $\eta_{\alpha,\beta}$ we have
\begin{align*}
    \eta_{\alpha,\beta}(x)
    &=\int_0^{\infty}\frac{h_{\alpha}(t)}{(t+x)^2}\, dt-\int_0^{\infty}\frac{h_{\alpha}(t)}{(t+\beta+x)^2}\, dt\\
    &=\int_0^{\beta}\frac{h_{\alpha}(t)}{(t+x)^2}\, dt+\int_{\beta}^{\infty}\frac{h_{\alpha}(t)-h_{\alpha}(t-\beta)}{(t+x)^2}\, dt,
\end{align*}
which is in $\mathcal S_2$ since $h_{\alpha}$ is non-decreasing (and non-negative).\end{proof}

Some symmetry and recursion relations are also needed. These are obtained by computation using the functional equation for the Gamma function:
\begin{lemma}
\label{lemma:reduction}
    The relations
    \begin{align}
     \phi_{a+1,b}(x)&=\phi_{a,b}(x)+\xi_{1,b}(x+a) \label{eq:reduction1}\\
     \phi_{a,b+1}(x)&=\phi_{a,b}(x)-\eta_{a,b}(x) \label{eq:reduction2}\\
     \phi_{b,a}(x)&=\phi_{a,b}(x) - \xi_{a,b}(x)\label{eq:reduction3} \\
     \phi_{b-1,a+1}(x)&=\phi_{a,b}(x)-\xi_{a,b-1}(x),\ b\geq 1\label{eq:reduction4}
    \end{align}
    hold for $x>0$.
\end{lemma}

\section{Proof of the main result}
As noticed in the introduction the function $g_{a,b}-bh_a$ may change its sign. It is easily seen that $g_{a,b}$ is continuous and bounded (and eventually $1$-periodic). Since $h_a$ is also bounded, integration by parts yields 
$$
\phi_{a,b}(x)
=2\int_0^{\infty}\frac{G_{a,b}(t)-bH_{a}(t)}{(t+x)^3}dt,
$$
where 
\begin{align*}
    G_{a,b}(x)&=\int_0^xg_{a,b}(t)\, dt,\\
    H_{a}(x)&=\int_0^xh_{a}(t)\, dt.
\end{align*}
The main part of the proof of Theorem \ref{thm:main} is contained in the next two lemmas. Their proofs are given after the proof of the main result.
\begin{lemma}
    \label{lemma:essential1} The function $\phi_{a,b}$ belongs to $\mathcal S_3$ when $0<a<b<1$.
\end{lemma}
\begin{lemma}
    \label{lemma:essential2} The function $-\phi_{a,b}$ belongs to $\mathcal S_3$ when $0<a<1$ and $1<b<a+1$.
\end{lemma}
\begin{rem}We have found it most intuitive to state the lemmas assuming strict inequalities between the parameters. This causes no restriction since the Stieltjes classes are closed under pointwise convergence. 
\end{rem}
\begin{proof}[Proof of Theorem \ref{thm:main}.]
   In the case where $0<b<1$ we argue as follows. Since $\xi_{1,b}$ belongs to $\mathcal S_2$ by Lemma \ref{lemma:xi} it follows from \eqref{eq:reduction1} that it is enough to show $\phi_{a,b}\in \mathcal S_3$ for $0<a,b<1$. If $a<b$ this is exactly the assertion of Lemma \ref{lemma:essential1}. If $a>b$ we use \eqref{eq:reduction3} to get 
   $$
   \phi_{a,b}(x)=\phi_{b,a}(x) +\xi_{a,b}(x).
   $$
   Since $\xi_{a,b}\in \mathcal S_2$ by Lemma \ref{lemma:xi} and $\phi_{b,a}\in \mathcal S_3$ by Lemma \ref{lemma:essential1} it follows that $\phi_{a,b}\in \mathcal S_3$ for $a> b$.

   In the situation where $b>1$ we let $\psi_{a,b}=-\phi_{a,b}$ and show that $\psi_{a,b}$ belongs to $\mathcal S_3$. From \eqref{eq:reduction2} and Lemma \ref{lemma:xi} it can be assumed that $1<b<2$. Furthermore, using \eqref{eq:reduction1} and Lemma \ref{lemma:xi} we may also assume $0<a<1$. If $b<1+a$ then the assertion follows from Lemma \ref{lemma:essential2}. If $b>1+a$  \eqref{eq:reduction4} is used to obtain
   $$
   \psi_{a,b}(x)=\psi_{b-1,a+1}(x)+\xi_{b-1,a}(x).
   $$
   Here, $a'=b-1\in (0,1)$, $b'=a+1\in (1,2)$ and $b'<a'+1$ so that (by Lemma \ref{lemma:essential2} again) and Lemma \ref{lemma:xi} we also have $\psi_{a,b}\in \mathcal S_3$.
\end{proof}

Before we give the proofs of Lemma \ref{lemma:essential1} and Lemma \ref{lemma:essential2} it is perhaps worthwhile to show the graphs of $1_{(0,a)}\ast1_{(0,b)}$ and $h_a$. See Figure \ref{fig:g-and-h}.
\begin{figure}
   \begin{center}
\begin{tabular}{cc}\includegraphics[scale=0.25]{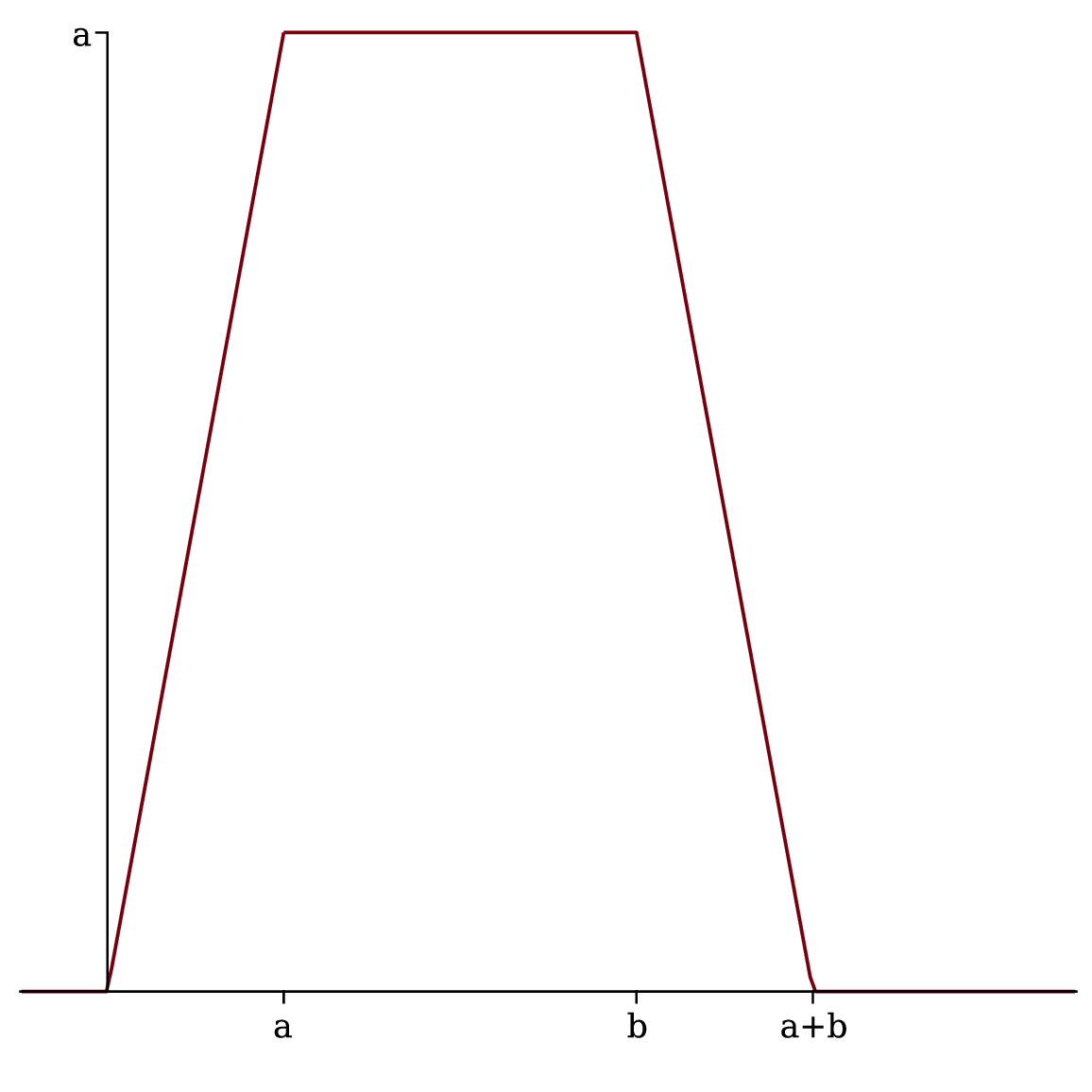} & \includegraphics[scale=0.25]{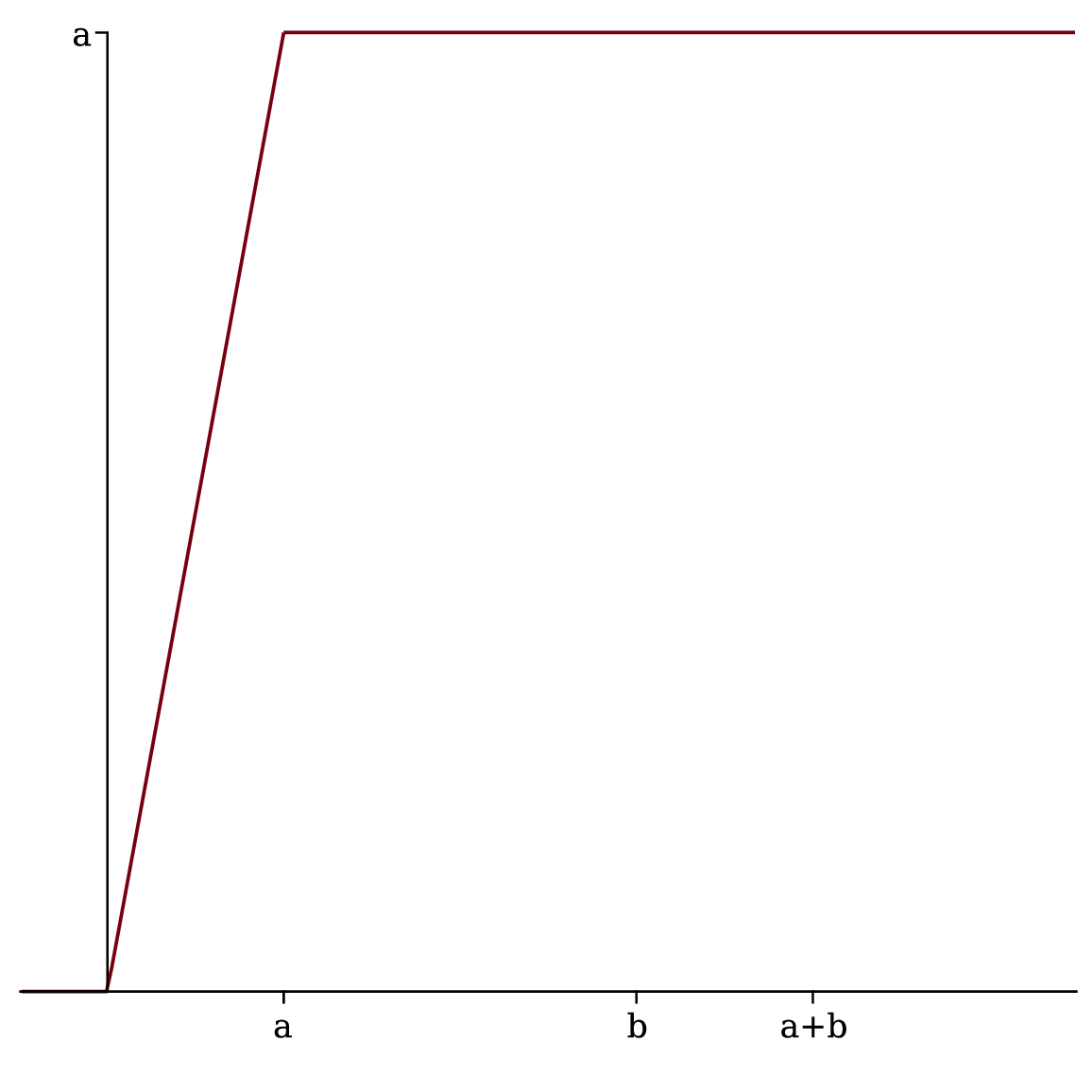}
\end{tabular}
\end{center}
\caption{$1_{(0,a)}\ast1_{(0,b)}$ and  $h_a$}
\label{fig:g-and-h}
\end{figure}

\begin{proof}[Proof of Lemma \ref{lemma:essential1}] To ease notation let $g=g_{a,b}$, $G=G_{a,b}$, $h=h_a$ and $H=H_a$. The function $g$ is $1$-periodic on $(a,\infty)$:
\begin{align*}
    g(t+1+a)&=\sum_{k=0}^{[t+1+a]}\left(1_{(0,a)}\ast1_{(0,b)}\right)(t+1+a-k)\\
&=\left(1_{(0,a)}\ast1_{(0,b)}\right)(t+1+a)+\sum_{k=0}^{[t+a]}\left(1_{(0,a)}\ast1_{(0,b)}\right)(t+a-k)\\
&=g(t+a),\quad\text{for}\ t\geq 0,
\end{align*}
since $t+1+a\geq 1+a\geq a+b$ so that $\left(1_{(0,a)}\ast1_{(0,b)}\right)(t+1+a)=0$. 

We proceed to verify that $G(x)-bH(x)\geq 0$ for $x\in [0,1+a]$. Since $h=g$ on $[0,b]$ it it clear that $G(x)-bH(x)\geq 0$ for $0\leq x\leq b$.

Elementary computations show
$$
H(x)=\left\{\begin{array}{ll}
     x^2/2,& 0\leq x\leq a, \\
     a(x-a/2),& x\geq a.
\end{array}\right.$$
In particular
\begin{align}
\label{eq:No1}
    G(a)-bH(a)&=(1-b)a^2/2,\\
    G(b)-bH(b)&=(1-b)(ab-a^2/2). \nonumber
\end{align} 
The computation of the integral $G(x)$ is more involved for $b\leq x\leq 1+a$. 
It is convenient to split the investigation into two cases, one where $a+b<1$, and another where $a+b>1$.

\begin{itemize}
\item First assume that  $a+b<1$.
For $x\in [b,a+b]$ we find, since $g(t)=a+b-t$ for $t\in [b,a+b]$ that 
\begin{align*}
G(x)-bH(x)&=G(b)-bH(b)+\int_b^x(a+b-t-ab)\, dt\\
&=G(b)-bH(b)+(a+b-ab)(x-b)-x^2/2+b^2/2.
\end{align*}
The minimal value of this polynomial on $[b,a+b]$ is taken at $x=b$ or at $x=a+b$. We have already seen that the value at $b$ is positive, and so is the value at $a+b$: a computation shows that
$$
G(a+b)-bH(a+b)=ab(1-b-a/2)>0,
$$
since $1-b>a$.

For $x\in [a+b,1]$, $G(x)-bH(x)=G(a+b)-bH(a+b)-ab(x-(a+b))\geq G(1)-bH(1)$ because $g$ is zero on this interval. Furthermore,  
$G(1)-bH(1)=a^2b/2>0$.

For $x\in[1,1+a]$ we have, using that $g(t)=g(t-1)$ on this interval,
\begin{align}
G(x)-bH(x)&=G(1)-bH(1)+\int_1^x(g(t)-ab)dt\nonumber \\
&=G(1)-bH(1)+(x-1)^2/2-ab(x-1).\label{eq:No2}
\end{align}
This expression takes its minimal value at $x=1+ab\in(1,1+a)$ and this value equals
$(1-b)ba^2/2$, which is positive.


\item Next consider the  case where $a+b>1$.
For $x\in [b,1]$ we have as before
$$
G(x)-bH(x)=G(b)-bH(b)+(a+b-ab)(x-b)-x^2/2+b^2/2.
$$
Its minimal value is taken at $x=b$ or at $x=1$. The value at $x=b$ is positive and a computation shows that
\begin{align*}
  G(1)-bH(1)&=-a^2/2-b^2/2-1/2+a^2b/2+a+b-ab.
\end{align*}
This is, by Lemma \ref{lemma:pos}, a positive expression. 

For $x\in [1,a+b]$, it is found that 
$$
G(x)-bH(x)=G(1)-bH(1)+(a+b-1-ab)(x-1).
$$
The assumption $a+b>1$ gives us $a+b-ab<1$ and therefore the expression above decreases as $x$ increases. Thus its minimal value is taken at $x=a+b$ and equals
\begin{equation}
\label{eq:pos-2}
G(a+b)-bH(a+b)=a^2/2+b^2/2+2ab-ab^2-a^2b/2-a-b+1/2.
\end{equation}
This, also, is positive by Lemma \ref{lemma:pos}. 

For $x\in[a+b,1+a]$ we have 
\begin{align}
G(x)-bH(x)&=G(a+b)-bH(a+b)+(x-1)^2/2\nonumber \\
&\phantom{=}\ -(a+b-1)^2/2-ab(x-a-b).\label{eq:No3}
\end{align}
This expression takes its minimal value at $x=1+ab\in(a+b,1+a)$ and this value is easily seen to equal the positive quantity
$a^2b(1-b)/2$.

\end{itemize}
For $x\geq 1+a$ the argument goes as follows. From the formulas \eqref{eq:No1}, \eqref{eq:No2} and \eqref{eq:No3} it follows that $G(1+a)-bH(1+a)=G(a)-bH(a)$, and this is important: Since  both $g$ and $h$ are $1$-periodic on $(a,\infty)$ we obtain
\begin{align*}
    G(x)-bH(x)&=G(1+a)-bH(1+a)+\int_{1+a}^x\left(g(t)-bh(t)\right)\, dt\\
    &=G(a)-bH(a)+\int_{a}^{x-1}\left(g(t)-bh(t)\right)\, dt\\
    &=G(x-1)-bH(x-1).
\end{align*}
Therefore, $G-bH$ is also $1$-periodic on $[a,\infty)$ and since it is positive on $[a,a+1]$ this finally yields the positivity of $G-bH$ on the positive line.\end{proof}
We denote $\tau_1$ and $\tau_2$ by
\begin{align*}
    \tau_1(a,b)&=-a^2/2-b^2/2-1/2+a^2b/2+a+b-ab,\\
    \tau_2(a,b)&=a^2/2+b^2/2+2ab-ab^2-a^2b/2-a-b+1/2.
\end{align*}
\begin{lemma}
    \label{lemma:pos}
    The functions $\tau_1$ and $\tau_2$ are positive in the region $\{(a,b)\, | \, 0<a<b<1, a+b>1\}$.
\end{lemma}
\begin{proof}
    Fix $b\in (1/2,1)$ and consider $\tau_1$ as a polynomial of the variable $a$
    $$
    \tau_1(a,b)=(1-b)(-a^2/2+a-(1-b)/2).
    $$
    It takes its maximum at $a=1$ and thus its minimum for $a\in (1-b,b)$ is attained at $a=1-b$. We find 
    $\tau_1(1-b,b)=b(1-b)^2/2$.
    By a similar argument it is found that $\tau_2(a,b)\geq b(1-b)^2/2$. 
\end{proof}
\begin{proof}[Proof of Lemma \ref{lemma:essential2}.] We assume that $0<a<1$ and $1<b<1+a$ and put $\psi=-\phi_{a,b}$ (and maintain the notation $g=g_{a,b}$, $G=G_{a,b}$, $h=h_{a}$ and $H=H_{a}$). Recalling the relation \eqref{eq:phi-int},
we aim at showing positivity of $bH-G$. The arguments are similar to those given in the proof of Lemma \ref{lemma:essential1}, and therefore we shall not give all details.

Computations show
$$
bH(x)-G(x)=
\left\{
\begin{array}{ll}
\tfrac{1}{2}(b-1)x^2, &x\in[0,a]\\
a(b-1)x-\tfrac{1}{2}(b-1)a^2, &x\in [a,1]\\
a(b-1)x-\tfrac{1}{2}(x-1)^2-\tfrac{1}{2}(b-1)a^2, &x\in [1,b]\\
(ab-a-b+1)x+\tfrac{1}{2}\left(a^2+b^2-1-ba^2\right), &x\in [b,1+a].
\end{array}
\right.$$
It is clear from this that $bH-G$ is positive on $[0,a]$, and that the minimum for $bH-G$ on $[a,1]$ is attained at $x=a$. The minimum of the quadratic in the third line above is taken at $x=1$ or at $x=b$. Here it is only needed to compute the value at $x=b$, which equals
$$
a(b-1)b-\tfrac{1}{2}(b-1)^2-\tfrac{1}{2}(b-1)a^2=(b-1)\left(ab-\tfrac{1}{2}(b-1)-\tfrac{1}{2}a^2\right).
$$
Using $b-1<a$ and $b>1$  one finds that this is a positive quantity. On the interval $[b,1+a]$ the minimum of $bH-G$ is attained at $x=1+a$. The minimal value  equals the positive value
$$
\tfrac{1}{2}a^2(b-1)+\tfrac{1}{2}(b-1)^2.
$$
Thus positivity is verified on the interval $[0,1+a]$.
The investigation for $x\in [1+a,1+b]$ is split into two cases. 
\begin{itemize} 
\item First the case where $a+b<2$ is considered. We find, for $x\in [1+a,a+b]$, 
\begin{align}
bH(x)-G(x)&=bH(1+a)-G(1+a)\nonumber\\
&+(ab-2a-b)(x-1-a)+\tfrac{1}{2}x^2-\tfrac{1}{2}(1+a)^2.\label{eq:same-formula}
\end{align}
This quadratic polynomial takes it minimal value at $x=2a+b-ab\in [1+a,a+b]$ and the value is (after some computation) found to equal the positive quantity
$$
\tfrac{1}{2}a(b-1)(2b-2+(2-b)a).
$$
For $x\in [a+b,2]$,
\begin{align*}
bH(x)-G(x)&=bH(a+b)-G(a+b)+\int_{a+b}^x(ba-a)\, dt.
\end{align*}
Since this expression increases with $x$, its minimal value is taken at $x=a+b$ and this value is larger than (as we saw above) the positive value $bH(2a+b-ab)-G(2a+b-ab)$.

For $x\in [2,1+b]$,
\begin{align*}
bH(x)-G(x)&=bH(2)-G(2)+\int_{2}^x(ba+2-a-t)\, dt\\
&=bH(2)-G(2)+(ab-a+2)x-\tfrac{1}{2}x^2-2ab+2a-2.
\end{align*}
The minimal value is attained at $x=2$ or at $x=1+b$. The value at $x=2$ is positive and a computation furthermore shows
 that $bH(1+b)-G(1+b)=bH(b)-G(b)$. Hence, $bH-G$ is positive on $[0,1+b]$.
\item We turn to the case $a+b>2$. For $x\in [1+a,2]$, \eqref{eq:same-formula} holds.
 As before, the minimal value, attained at $x=2a+b-ab$, is positive. We next consider the behavior on the intervals $[2,a+b]$ and $[a+b,1+b]$.
For $2\leq x\leq a+b$,
$$
bH(x)-G(x)=bH(2)-G(2)+(ab-2a-b+2)(x-2),$$
and since this expression increases with $x$ its minimal value is attained at $x=2$. From the investigation on the interval $[1+a,2]$ we obtain that this value is positive. (And by computation it is seen that 
$$
bH(2)-G(2)=a^2+2-2b-\tfrac{1}{2}a^2b+\tfrac{1}{2}b^2+2ab-3a.) 
$$ 
For $a+b\leq x\leq 1+b$,
\begin{align*}
bH(x)-G(x)&=bH(a+b)-G(a+b)\\
&+(ab-a+2)(x-a-b)-\tfrac{1}{2}x^2+\tfrac{1}{2}(a+b)^2.
\end{align*}
Its minimal value is taken at $x=a+b$ or at $x=1+b$. We know form above that the value at $x=a+b$ is positive. A computation shows that the value at $x=1+b$ is the same as the value at $x=b$.
\end{itemize}
We have thus obtained the positivity of $bH-G$ on the interval $[0,1+b]$. Positivity on the entire half-line follows using the $1$-periodicity of $bh-g$ on $(b,\infty)$ in the same way as in the proof of Lemma \ref{lemma:essential1}.
\end{proof}

\noindent
Henrik Laurberg Pedersen\\
Department of Mathematical Sciences\\
University of Copenhagen\\
Universitetsparken 5\\
DK-2100, Denmark\\
email: henrikp@math.ku.dk
\end{document}